\documentclass[a4,12pt]{article}%

\usepackage{graphicx}
\usepackage{graphics}

\usepackage{amsmath}
\usepackage{amsfonts}
\usepackage{amssymb}
\usepackage{enumerate}

\newtheorem{theorem}{Theorem}[section]

\newtheorem{example}[theorem]{Example}

\newtheorem{lemma}[theorem]{Lemma}

\newtheorem{proposition}[theorem]{Proposition}

\newenvironment{proof}[1][Proof]{\textbf{#1.} }{\ \rule{0.5em}{0.5em}}

\newcommand{\dN}{{\bf N}}

\newcommand{\dR}{{\bf R}}

\newcommand{\calT}{{\cal T}}

\newcommand{\ep}{\varepsilon}

\newcounter{figurecounter}

\begin{document}

\title{Browder's Theorem through Brouwer's Fixed Point Theorem%
\thanks{The first author acknowledges the support of the Israel Science Foundation, Grant \#217/17.}}

\author{Eilon Solan and Omri N.~Solan%
\thanks{The School of Mathematical Sciences, Tel Aviv
University, Tel Aviv 6997800, Israel. e-mail: eilons@post.tau.ac.il, omrisola@post.tau.ac.il.}}

\maketitle

\begin{abstract}
One of the conclusions of Browder (1960) is a parametric version of Brouwer's Fixed Point Theorem,
stating that for every continuous function $f : ([0,1] \times X) \to X$,
where $X$ is a simplex in a Euclidean space,
the set of fixed points of $f$, namely, the set $\{(t,x) \in [0,1] \times X \colon f(t,x) = x\}$,
has a connected component whose projection on the first coordinate is $[0,1]$.
Browder's (1960) proof relies on the theory of the fixed point index.
We provide an alternative proof to Browder's result using Brouwer's Fixed Point Theorem.
\end{abstract}


\bigskip


\section{Introduction}

Brouwer's Fixed Point Theorem (Hadamard, 1910, Brouwer, 1911) states that every continuous function from a finite dimensional simplex into itself
has a fixed point.
This result was later generalized to nonempty, convex, and compact subsets of more general topological vector spaces, see, e.g.,
Schauder (1930), Tychonoff (1935), and Dyer (1956).

The following parametric version of Brouwer's Fixed Point Theorem
is a special case of a more general result of
Browder (1960).
To state the theorem we need the concept of connected component.
A set $A \subseteq \dR^n$ is \emph{connected} if there are no two disjoint open sets $O_1,O_2$ that satisfy (a) $A \subseteq O_1 \cup O_2$,
(b) $A \not\subseteq O_1$, and
(c) $A \not\subseteq O_2$.
A subset $B$ of $A$ is a \emph{connected component} of $A$ if every connected subset of $A$ is either contained in $B$ or disjoint of $B$.

\begin{theorem}[Browder, 1960]
\label{theorem:browder}
Let $f : ([0,1] \times X) \to X$ be a continuous function,
where $X = [0,1]^n$.
Define
the set of fixed points of $f$ by
\begin{equation}
\label{equ:98}
C_f := \{ (t,x) \in [0,1] \times X \colon f(t,x)=x\}.
\end{equation}
Then $C_f$ has a connected component whose projection to the first coordinate is $[0,1]$.
\end{theorem}

\begin{example}
\label{example:1}
Let $X = [-1,1]$, and $f : [0,1] \times X \to X$ be given by
\[ f(t,x) := \left\{
\begin{array}{lll}
x & \ \ \ \ \ & t=0,\\
(1-t)x + t\sin(\tfrac{1}{t}) & & t \neq 0.
\end{array}
\right.
\]
The set $C_f$ is the union of $\{ (0,x) \colon x \in [-1,1]\}$ and
$\{ (t, \sin(\tfrac{1}{t})) \colon t \in (0,1]\}$, which is connected (but not path connected).
\end{example}

Theorem~\ref{theorem:browder} was used in a variety of topics,
like nonlinear complementarity theory (see, e.g., Eaves, 1971, or Allgower and Georg, 2012),
nonlinear elliptic boundary value problems
(Shaw, 1977),
the study of global continua of solutions of nonlinear
partial differential equations (see, e.g., Costa and Gon\c{c}alves, 1981, or Massabo and Pejsachowitz, 1984),
theoretical economics (Citanna et al., 2001),
and game theory (see, e.g., Herings and Peeters, 2010, or Solan and Solan, 2021).

Browder's (1960) proof of Theorem~\ref{theorem:browder} uses the fixed point index,%
\footnote{In fact, the statement of Browder's (1960) more general version of Theorem~\ref{theorem:browder}
is phrased using the fixed point index.}
and hence is not accessible to many researchers,
and cannot be taught in an undergraduate course in topology.
In this paper we prove Theorem~\ref{theorem:browder} using Brouwer's Fixed Point Theorem.
In particular, our proof is accessible to all mathematicians,
and can be taught in any course in which Brouwer's Fixed Point Theorem
is proved.
In Section~\ref{section:discussion} we discuss extensions of Theorem~\ref{theorem:browder}
to more general parameter sets and more general sets $X$.

Theorem~\ref{theorem:browder} easily follows from Brouwer's Fixed Point Theorem when the number of connected components
of $C_f$ is finite.%
\footnote{In this case, our Proposition~\ref{prop:1} is trivial,
hence the proof reduces to our Proposition~\ref{prop:2}.}
As the next example shows,
the number of connected components of $C_f$ may not be finite or countable.

\begin{example}
Recall that the \emph{Cantor set} $K$ is the set of all real numbers in $[0,1]$ such that,
in their representation in base 3, appear only the digits 0 and 2.
The cardinality of the Cantor set is the continuum,
and its complement is a union of countably many open intervals.
Let $g : [0,1] \to [0,1]$ be the function that is the identity on $K$,
and, on each maximal open subinterval $(a,b)$ of $[0,1]$ in the complement of $K$ it is given by
$g(x) = x + (x-a)(b-x)$, see Figure~\arabic{figurecounter}.
The function $g$ is continuous, its range is $[0,1]$, and its set of fixed points is $K$.

Define now a function $f : [0,1] \times [0,1] \to [0,1]$ by
$f(t,x) = g(x)$ for every $(t,x) \in [0,1] \times [0,1]$.
The connected components of $C_f$ are then all sets of the form $[0,1] \times \{x\}$ for $x \in K$.
\label{example:2}
\end{example}

\bigskip
\centerline{ \includegraphics{browder_fig.1} }

\vspace{0.2truecm}

\centerline{{Figure
\arabic{figurecounter}:
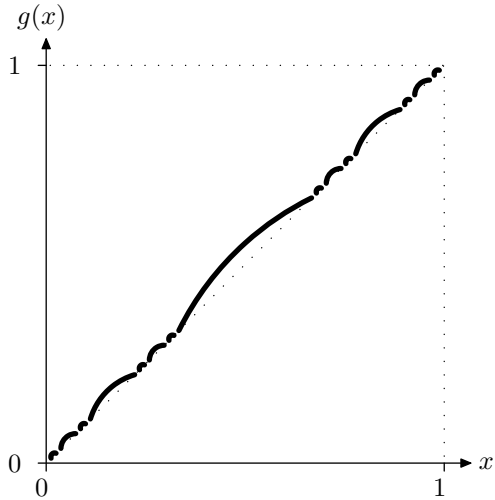
The function $g$ in Example~\ref{example:2}.}}
\addtocounter{figurecounter}{1}

\section{Proof of Theorem~\ref{theorem:browder}}

In this section we prove Theorem~\ref{theorem:browder}.
The theorem will follow from Brouwer's Fixed Point Theorem once we prove the following two results.

\begin{proposition}
\label{prop:1}
If $C_f$ has no connected component whose projection on the first coordinate is $[0,1]$,
then there are
two disjoint open sets $O_0$ and $O_1$ that satisfy the following properties:
\begin{enumerate}
\item[(C1)]
The sets $O_0$ and $O_1$ cover $C_f$, that is, $C_f \subseteq O_0 \cap O_1$.
\item[(C2)]
Every connected component $B$ of $C_f$ that satisfies $B \cap (\{0\} \times X) \neq \emptyset$ is a subset of $O_0$.
\item[(C3)]
Every connected component $B$ of $C_f$ that satisfies $B \cap (\{1\} \times X) \neq \emptyset$ is a subset of $O_1$.
\end{enumerate}
\end{proposition}

\begin{proposition}
\label{prop:2}
If there are
two disjoint open sets $O_0$ and $O_1$ that satisfy (C1)--(C3),
then there is a continuous function $F : ([0,1] \times X) \to ([0,1] \times X)$
that has no fixed point.
\end{proposition}

To see that the two propositions imply Theorem~\ref{theorem:browder},
note that the conclusion of Proposition~\ref{prop:2} contradicts Brouwer's Fixed Point Theorem,
hence Proposition~\ref{prop:2} implies that there are no
two disjoint open sets $O_0$ and $O_1$ that satisfy (C1)--(C3).
Hence Proposition~\ref{prop:1} implies that $C_f$ has a connected component whose projection on the first coordinate is $[0,1]$.

As Example~\ref{example:1} shows, connected components of $C_f$ may be complicated sets,
and as Example~\ref{example:2} shows, the number of connected components of $C_f$ may be of the order of the continuum.
In particular, the condition that
$C_f$ has no connected component whose projection on the first coordinate is $[0,1]$
is difficult to use.
Proposition~\ref{prop:1} turns the contrapositive assumption of Theorem~\ref{theorem:browder}
into a seemingly stronger condition that is easier to use.

The proof of Proposotion~\ref{prop:1} is the more challenging part of the
proof of Theorem~\ref{theorem:browder},
and it goes through the following steps.
\begin{itemize}
	\item
	For every $k \in \dN$ we will approximate $C_f$ by a ``simple'' set $S_k$.
	We will do that by covering $[0,1] \times X$ with finitely many boxes of diameter $\frac{1}{2^k}$, and defining $S_k$ to be the union of all boxes that intersect $C_f$.
	\item
	We will then prove that if $C_f$ has no connected component whose projection on the first coordinate is $[0,1]$,
	then there is $k$ such that $S_k$ has no connected component whose projection on the first coordinate is $[0,1]$.
	\item
	Since $S_k$ is the union of finitely many boxes,
	if it has no connected component whose projection on the first coordinate is $[0,1]$,
	then the existence of two disjoint open sets $O_0$ and $O_1$ that satisfy (C1)--(C3) w.r.t.~$S_k$ (rather than $C_f$) is clear. 
	Since $S_k \supseteq C_f$, Proposition~\ref{prop:1} follows.
\end{itemize}

\subsection{Proof of Proposition~\ref{prop:1}}

The maximum norm in $\dR^n$ is given by
\[ d_\infty(y,y') := \max_{i=1,\dots,n}|y_i-y'_i|, \ \ \ \forall y,y' \in \dR^n. \]
In the proof we will use the distance between a point and a set and the distance between two sets:
for every two sets $A,A' \subseteq \dR^n$ and every point $y \in \dR^n$,
\[ d_\infty(y,A) := \inf_{y' \in A} d_\infty(y,y'),
\ \ \ \ \
d_\infty(A,A') := \inf_{y \in A, y' \in A'} d_\infty(y,y'). \]
We will also use the Hausdorff distance between sets:
\[ d_H(A,A') := \max\left\{ \sup_{y \in A} d_\infty(y,A'), \sup_{y' \in A'} d_\infty(y',A)\right\}. \]


For every $k \in \dN$, let $\calT_k$ be the collection of all
boxes $\prod_{i=1}^{n+1} [a_i,b_i] \subseteq [0,1]^{n+1}$ where $a_i$ and $b_i$ are rational numbers that are integer multiples of $\frac{1}{2^k}$.
We note that $\calT_{k+1}$ refines $\calT_k$:
every set $T \in \calT_{k+1}$ is a subset of some set $T' \in \calT_k$.
Let
\[ S_k := \bigcup\{ T \in \calT_k \colon T \cap C_f \neq \emptyset\}. \]
This is the union of all boxes in $\calT_k$ that contain points in $C_f$, see Figure~\arabic{figurecounter}, where the set $C_f$ has three connected components.
In particular, $S_k \supseteq C_f$.
Since $\calT_{k+1}$ refines $\calT_k$, we have $S_{k+1} \subseteq S_k$.

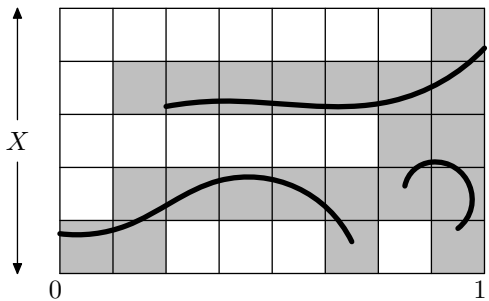
\bigskip
\centerline{ \includegraphics{browder_fig.6} }

\vspace{0.2truecm}

\centerline{{Figure
\arabic{figurecounter}:
The boxes in $\calT_k$ and the sets $C_f$ (dark) and $S_k$ (grey).}}
\addtocounter{figurecounter}{1}

\begin{lemma}
If there is $k \in \dN$ such that $S_k$ has no connected component whose projection on
the first coordinate is $[0,1]$,
then there are two disjoint open sets $O_0$ and $O_1$
that satisfy (C1)--(C3).
\end{lemma}

\begin{proof}
Let $A_0$ be the union of all connected component of $S_k$ that intersect $\{0\} \times X$.
Let $A_1 := S_k \setminus A_0$.
Each of the sets $A_0$ and $A_1$ is a union of finitely many boxes,
and the two sets are disjoint.
It follows that $d_\infty(O_0,O_1) \geq \frac{1}{2^k}$.
This implies that the open sets
\[ O_0 := \left\{ (t,x) \in Y \colon d_\infty((t,x),A_0) < \frac{1}{2^{k+1}}\right\}, \]
and
\[ O_1 := \left\{ (t,x) \in Y \colon d_\infty((t,x),A_1) < \frac{1}{2^{k+1}}\right\}, \]
satisfy (C1)--(C3).
\end{proof}

\bigskip

From now on we assume that for every $k \in \dN$, the set $S_k$ has a connected component
whose projection on the first coordinate is $[0,1]$.
We will prove that in this case,
$C_f$ has a connected component whose projection on the first coordinate is $[0,1]$.

\begin{lemma}
Suppose that for every $k \in \dN$, the set $S_k$ has a connected component
whose projection on the first coordinate is $[0,1]$.
There is a decreasing sequence of closed sets $(D_k)_{k \in \dN}$ that satisfies the following properties for every $k \in \dN$:
\begin{enumerate}
\item[(D1)] $D_k$ is a union of boxes in $\calT_k$, and in particular it is closed.
\item[(D2)] $D_k \subseteq S_k$.
\item[(D3)] If $k > 1$ then $D_k \subseteq D_{k-1}$.
\item[(D4)] For every $l \geq k$, The set $S_l \cap D_k$ has a connected component
whose projection on the first coordinate is $[0,1]$.
\item[(D5)] $D_k$ is a minimal subset of $S_k$ that satisfies (D1)--(D4).
\end{enumerate}
\end{lemma}

\begin{proof}
The proof is by induction over $k$.
Define $D_{-1} := [0,1] \times X$,
let $k \in \dN$ be given, and assume that we already defined $(D_j)_{j=1}^{k-1}$
in a way that satisfies (D1)--(D5).
We argue that the set $D_k := S_k \cap D_{k-1}$ satisfies (D1)--(D4).
For $k=1$ this holds by the properties of $(S_k)_{k\in \dN}$.

Assume now that $k>1$.
By definition, (D2) and (D3) hold.
Since $S_k$ is a union of boxes in $\calT_k$,
since $D_{k-1}$ is a union of boxes in $\calT_{k-1}$,
and since $\calT_k$ refines $\calT_{k-1}$, (D1) holds.
(D4) holds since $D_{k-1}$ satisfies (D4).

Since the set $S_k \cap D_{k-1}$ is composed of finitely many boxes in $\calT_k$,
it has finitely many subsets that satisfy (D1)--(D3),
and at least one of them, $S_k \cap D_{k-1}$, satisfies (D4).
Let $D_k$ be a minimal (w.r.t.~set inclusion) subset of $S_k \cap D_{k-1}$
that satisfies (D1)--(D4).
Then $D_k$ also satisfies (D5).
\end{proof}

\bigskip

Since the sequence $(D_k)_{k \in \dN}$ is a decreasong sequence of closed sets,
the intersection
$D_* := \bigcap_{k \in \dN} D_k$ is closed and nonempty.
Since the sets $(D_k)_{k \in \dN}$ are contained in the compact set $[0,1]\times X$,
we have $\lim_{k \to \infty} d_H(D_k,D_*) = 0$.

As we now show, the minimality of $D_k$ (Property (D5)) implies that $D_k$ is connected.
In fact, this implication is the reason for requiring $D_k$ to satisfy Property~(D5).

\begin{lemma}
\label{lemma:2a}
The set $D_k$ is connected, for every $k \in \dN$.
\end{lemma}

\begin{proof}
Assume by way of contradiction that $D_{k_*}$ is not connected for some $k_* \in \dN$.
Let $O'$ and $O''$ be two disjoint open sets that satisfy
(a) $D_{k_*} \subseteq O' \cup O''$,
(b) $D_{k_*} \not\subseteq O'$, and
(c) $D_{k_*} \not\subseteq O''$.

For every $k \in \dN$, every connected component of $D_k$ lies either in $D_k \cap O'$ or in $D_k \cap O''$.
Hence, and since for every $k \in \dN$,
the set $D_k$ has a connected component whose projection on the first coordinate is $[0,1]$,
at least one of the sets $D_k \cap O'$ and $D_k \cap O''$
has a connected component whose projection on the first coordinate is $[0,1]$.
Assume w.l.o.g.~that for infinitely many $k$'s,
the set $D_k \cap O'$
has a connected component whose projection on the first coordinate is $[0,1]$.

Since the sequence $(D_k)_{k \in \dN}$ is decreasing,
if $D_{k+1} \cap O'$
has a connected component whose projection on the first coordinate is $[0,1]$,
then $D_k \cap O'$
has a connected component whose projection on the first coordinate is $[0,1]$.
It follows that for every $k \in \dN$, and in particular for $k=k_*$,
the set $D_k \cap O'$
has a connected component whose projection on the first coordinate is $[0,1]$.

But then the set $D_{k_*} \cap O'$ satisfies Properties (D1)--(D4) for $k_* = k$,
contradicting the minimality of $D_{k_*}$ (Property (D5)).
\end{proof}

\begin{lemma}
\label{lemma:2}
The set $D_*$ is connected.
\end{lemma}

\begin{proof}
Assume by way of contradiction that $D_*$ is not connected,
and let $O'$ and $O''$ be two disjoint open sets that satisfy
(a) $D_* \subseteq O' \cup O''$,
(b) $D_* \not\subseteq O'$, and
(c) $D_* \not\subseteq O''$.

We have $D_* \cap O' = D_* \cap (O'')^c$,
hence $D_* \cap O'$ is closed, and it is disjoint of the closed set $(O')^c$.
It follows that $d_\infty(D_* \cap O',(O')^c) > 0$.
Similarly, $d_\infty(D_* \cap O'',(O'')^c) > 0$.

Since $\lim_{k \to \infty} d_H(D_k,D_*) = 0$,
It follows that the sets $O'$ and $O''$ satisfy (a)--(c) w.r.t.~$D_k$ (instead of w.r.t.~$D_*$),
for every $k$ sufficiently large.
This implies that for every such $k$, the set $D_k$ is not connected,
contradicting Lemma~\ref{lemma:2a}.
\end{proof}

\bigskip

For every $k \in \dN$, the set $D_k$ has a connected component whose projection on the first coordinate is $[0,1]$.
In particular, for every $t \in [0,1]$ there is $x_{k,t} \in X$ such that $(t,x_{k,t})\in D_k$.
Since $X$ is compact, the sequence $(x_{k,t})_{k \in \dN}$ has a converging subsequence.
Since $\lim_{k \to \infty} d_H(D_K,D_*) = 0$, it follows that there is $x_t \in X$ such that $(t,x_t) \in D_*$.
Therefore, the projection of $D_*$ on the first coordinate is $[0,1]$.
It follows that the connected component of $C_f$ that contains $D_*$
satisfies the property that its
projection on the first coordinate is $[0,1]$,
contradicting the assumption in the proposition.

\subsection{Proof of Proposition~\ref{prop:2}}

We start by defining a continuous function $g : ([0,1] \times X) \to [-1,1]$ that satisfies the following properties (see Figure~\arabic{figurecounter},
where $C_f$ has six connected components):
\begin{itemize}
\item   $g \equiv 1$ on $B_1 := (\{0\} \times X) \cup (C_f \cap O_0)$.
\item   $g \equiv -1$ on $B_{-1} := (\{1\} \times X) \cup (C_f \cap O_1)$.
\end{itemize}

\bigskip
\centerline{ \includegraphics{browder_fig.2} \ \ \ \ \ \ \ \ \ \ \includegraphics{browder_fig.3}}

\vspace{0.2truecm}

\centerline{{Figure
\arabic{figurecounter}:
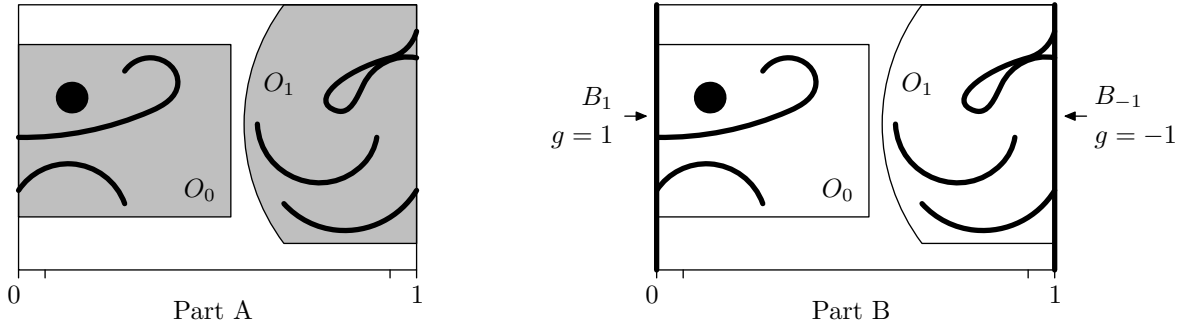
The sets $C_f$ (dark), $O_0$ and $O_1$ (grey) in Part A;}}
\centerline{{the sets $B_1$, $B_{-1}$ (dark) in Part B.}}
\addtocounter{figurecounter}{1}

\bigskip

The sets
$B_1$
and
$B_{-1}$
are disjoint.
As in the proof of Lemma~\ref{lemma:2},
$C_f\cap O_0 = C_f \cap (O_1)^c$ is closed, 
and similarly, $C_f \cap O_1$ is closed.
It follows that
$B_1$
and
$B_{-1}$
and closed, and hence by Titze's Extension Theorem such a function $g$ exists,
for example,
\[ g(t,x) := \left\{
\begin{array}{lll}
1, & \ \ \ \ \ & (t,x) \in B_1,\\
-1, & \ \ \ \ \ & (t,x) \in B_{-1},\\
\frac{d_\infty\bigl((t,x),B_{-1}\bigr) - d_\infty\bigl((t,x),B_1\bigr)}{d_\infty\bigl((t,x),B_{-1}\bigr) + d_\infty\bigl((t,x),B_1\bigr)},
& & \hbox{Otherwise}.
\end{array}
\right. \]

We argue that for every $\ep > 0$ sufficiently small,
$t + \ep g(t,x) \geq 0$ for every $(t,x) \in [0,1] \times X$.
Indeed, since $g$ is continuous over the compact set $[0,1] \times X$,
it is absolutely continuous.
Hence, and since $g \equiv 1$ on $\{0\} \times X$,
it follows that there is $\ep > 0$ such that $g(t,x) \geq 0$
whenever $t \leq \ep$.
This implies that if $t \leq \ep$ then
\[ t + \ep g(t,x) \geq t \geq 0, \]
and if $t \geq \ep$, then since $g(t,x) \geq -1$ we have
\[ t + \ep g(t,x) \geq t - \ep \geq 0. \]
Analogously, for every $\ep > 0$ sufficiently small
$t + \ep g(t,x) \leq 1$ for every $(t,x) \in [0,1] \times X$.

Let $\ep > 0$ be sufficiently small so that $t + \ep g(t,x) \in [0,1]$ for every $(t,x) \in [0,1] \times X$.
Consider the function $F : ([0,1] \times X) \to ([0,1] \times X)$ defined by
\[ F(t,x) := \bigl(t + \ep g(t,x), f(t,x)\bigr), \ \ \ \forall (t,x) \in [0,1] \times X. \]
The function $F$ is continuous, and by the choice of $\ep$ its range is indeed $[0,1] \times X$.
We argue that $F$ has no fixed point.
Indeed, if $(t^*,x^*)$ is a fixed point of $F$,
then
\[ t^* = t^* + \ep g(t^*,x^*), \ \ \ x^* =  f(t^*,x^*). \]
This implies that $g(t^*,x^*) = 0$.
Since $g$ attains the values 1 and $-1$ on $C_f$, it follows that $(t^*,x^*) \not\in C_f$.
On the other hand, since $x^* =  f(t^*,x^*)$, we have $(t^*,x^*) \in C_f$,
a contradiction.

\section{Extensions}
\label{section:discussion}

We proved Theorem~\ref{theorem:browder} when $X = [0,1]^n$.
Our proof holds whenever $X$ is a nonempty, convex, and compact subset
of a locally convex metrizable topological vector space.
The only part of the proof that needs to be adapted for this extension
is the definition of $\calT_k$.
Since $[0,1] \times X$ is compact and metrizable,
for every $k \in \dN$ there is a finite collection $(T_{k,l})_{l=1}^{L_k}$ of open sets
with diamater smaller than $\frac{1}{k}$ that covers $X$.
We can assume furthermore that each open set $T_{k+1,l}$ is a subset of
one of the sets $(T_{k,l})_{l=1}^{L_k}$.
We then define $\calT_k$ to be the collection of closures of $(T_{k,l})_{l=1}^{L_k}$,
for each $k \in \dN$.

We note that Browder's (1960) proof using the fixed point index 
implies Theorem~\ref{theorem:browder} when $X$ is a nonempty, convex, and compact subset
of a locally convex topological vector space (but not necessarily metrizable).

\bigskip

In Theorem~\ref{theorem:browder},
the parameter set is $[0,1]$.
One may wonder whether the theorem remains valid for more general parameter sets.
The answer is positive.
We here illustrate this extension for the parameter set $[0,1]^2$.

Let $f : ([0,1]^2 \times X) \to X$ be a continuous function,
where $X = [0,1]^n$,
and let $\varphi : [0,1] \to [0,1]^2$ be a continuous and surjective function
(e.g., the Peano's curve (Peano, 1890)).
The function $h := f \circ (\varphi,{\mathrm Id}_X) : [0,1] \times X \to X$ is a composition of two continuous functions, hence continuous, and by Theorem~\ref{theorem:browder} the set $C_h$ has a connected component, denoted $B$, whose projection on the first coordinate is $[0,1]$.
But then the set $\{(\varphi(t),x) \colon (t,x) \in B\}$ is a connected component of $C_f$
whose projection on the first coordinate is $[0,1]^2$.

Note that this construction is valid whenever the parameter set $Y$ possesses a space-filling curve,
namely, there is a continuous and surjective function $\varphi : [0,1] \to Y$.
Recall that the Hahn-Mazurkiewicz Theorem (e.g., Willard, 2012, Theorem~31.5) states that a space possesses a space-filling curve
if and only if it is compact, connected, locally connected, and second-countable.
One example of a set that does not possess a space-filling curve is the set $C_f$ in Example~\ref{example:1}.

\end{document}